\documentclass[11pt]{amsart}
\usepackage{amsaddr}
\usepackage{amssymb}
\usepackage{hyperref}
\usepackage{color}
\usepackage[margin=1in]{geometry}

\newtheorem{theorem}{Theorem}[section]
\newtheorem{lemma}[theorem]{Lemma}
\newtheorem{proposition}[theorem]{Proposition}
\newtheorem{corollary}[theorem]{Corollary}
\newtheorem{conjecture}[theorem]{Conjecture}

\newtheorem*{Euler}{Euler's Formula}

\theoremstyle{definition}

\newtheorem{example}[theorem]{Example}

\theoremstyle{remark}
\newtheorem{remark}[theorem]{Remark}

\begin{document}

\title[New upper bounds for the bondage number of a graph]{New upper bounds for the bondage number of a graph in terms of its maximum degree and Euler characteristic}

\author{Jia Huang}
\address{Department of Mathematics and Statistics \\ University of Nebraska at Kearney \\ Kearney, NE 68849, USA}
\email{huangj2@unk.edu}

\author{Jian Shen}
\address{Department of Mathematics \\ Texas State University \\ San Marcos, TX 78666, USA}
\email{js48@txstate.edu}

\thanks{The authors thank the anonymous referee for providing helpful suggestions.}

\maketitle

\begin{abstract}
The bondage number $b(G)$ of a graph $G$ is the smallest number of edges whose removal from $G$ results in a graph with larger domination number. Let $G$ be embeddable on a surface whose Euler characteristic $\chi$ is as large as possible, and assume $\chi\leq0$. Gagarin--Zverovich and Huang have recently found upper bounds of $b(G)$ in terms of the maximum degree $\Delta(G)$ and the Euler characteristic $\chi$. In this paper we prove a better upper bound $b(G)\leq\Delta(G)+\lfloor t\rfloor$ where $t$ is the largest real root of the cubic equation $z^3 + z^2 + (3\chi - 8)z  + 9\chi - 12=0$; this upper bound is  asymptotically equivalent to $b(G)\leq\Delta(G)+1+\lfloor \sqrt{4-3\chi} \rfloor$. We also establish further improved upper bounds for $b(G)$ when the girth, order, or size of the graph $G$ is large compared with $|\chi|$.
\end{abstract}

\section{Introduction}

The graphs considered in this paper are all finite, undirected, and without loops or multiple edges. Let $G$ be a graph with vertex set $V(G)$ and edge set $E(G)$. The \emph{order} and \emph{size} of $G$ are $|V(G)|$ and $|E(G)|$, respectively. A graph with order $1$ is called \emph{trivial} and a graph with size $0$ is called \emph{empty}. The \emph{degree} $d(v)$ of a vertex $v$ in $G$ is the cardinality $|N(v)|$ of the set $N(v)$ of all neighbors of $v$ in $G$. The maximum and minimum vertex degrees
of $G$ are denoted by $\Delta(G)$ and $\delta(G)$.

A \emph{dominating set} of a graph $G$ is a subset $D\subseteq V$ of vertices such that every vertex not in $D$ is adjacent to at least one vertex in $D$. The minimum cardinality of a dominating set is called the \emph{domination number} of the graph $G$. The concept of domination in graphs has many applications in a wide range of areas. 

The \emph{bondage number} $b(G)$ of a graph $G$, introduced in \cite{BHNS,FJKR}, is the smallest number of edges whose removal from $G$ results in a graph with larger domination number. It measures to some extent the reliability of the domination number of the graph $G$, as edge removal from $G$ corresponds to link failure in a communication network whose underlying structure is given by the graph $G$. One can check that the bondage number $b(G)$ is well defined for any nonempty graph $G$.

In general it is NP-hard to determine the bondage number $b(G)$ (see Hu and Xu~\cite{HuXu}). There have been studies on its upper and lower bounds, such as the following results.

\begin{lemma}[Hartnell and Rall~\cite{HR1}]\label{lem:HR}
For any edge $uv$ in a graph $G$ one has
\[
b(G) \leq d(u) + d(v)-1-|N(u) \cap N(v)|.
\]
In particular, $b(G)\leq \Delta(G) + \delta(G)-1$.
\end{lemma}

\begin{theorem}[Hartnell and Rall~\cite{HR2}]\label{thm:ad}
For any connected graph $G$ one has a sharp bound $|E(G)| \geq |V(G)|(b(G) + 1)/4$.
\end{theorem}

The \emph{average degree} of a graph $G$ is defined as $ad(G):=2|E(G)|/|V(G)|$. It follows from Lemma~\ref{lem:HR} and Theorem~\ref{thm:ad} that $b(G)\leq b'(G)$ where $b'(G)$ is an integer defined as
\begin{equation}\label{eq:b'}
b'(G) = \min\left\{ \min_{uv\in E(G)} d(u) + d(v)-1-|N(u) \cap N(v)|,\ 2 \lfloor ad(G) \rfloor-1 \right\}.
\end{equation}
Any upper bound for $b'(G)$ is automatically an upper bound for the bondage number $b(G)$.

There are two conjectures on upper bounds of $b(G)$, which are still open.

\begin{conjecture}[Teschner~\cite{Teschner}]\label{conj1}
For any graph $G$ we have $b(G)\leq \frac32 \Delta(G)$.
\end{conjecture}

\begin{conjecture}[Dunbar-Haynes-Teschner-Volkmann~\cite{DHTV}]
\label{conj2}
If $G$ is a planar graph then we have $b(G)\leq \Delta(G)+1$.
\end{conjecture}

The best upper bound known so far for the bondage number of a planar graph is the following.

\begin{theorem}[Kang and Yuan~\cite{KangYuan}]
For any planar graph $G$, $b(G)\leq\min\{\Delta(G)+2,8\}$.
\end{theorem}

Carlson and Develin~\cite{CarlsonDevelin} provided a simpler proof for the above theorem, which was further extended by Gagarin and Zverovich~\cite{GZ1} to establish a nice upper bound for arbitrary graphs, a step forward towards Conjecture~\ref{conj1}. The main idea is to embed graphs on surfaces, which we outline next. The readers are referred to Mohar and Thomassen~\cite{MoharThomassen} for more details on graph embedding.

Throughout this paper a \emph{surface} means a connected compact
Hausdorff topological space which is locally homeomorphic to an open disc in $\mathbb R^2$. According to the classification theorem for surfaces~\cite[Theorem 3.1.3]{MoharThomassen}, any surface $S$ is homeomorphic to either $S_h$ ($h\geq0$) which
is obtained from a sphere by adding $h$ handles, or $N_k$ ($k\geq1$) which is obtained from a sphere by adding $k$ crosscaps. In the former case $S$ is an \emph{orientable surface of genus $h$}, and in the latter case $S$ is a \emph{non-orientable surface of genus $k$}. For instance, the torus, the projective plane, and the Klein bottle are homeomorphic to $S_1$, $N_1$, and $N_2$, respectively. The \emph{Euler characteristic} of the surface $S$ is defined as
\[
\chi(S)=\begin{cases}
         2-2h, & {\rm if}\ S\cong S_h,\\
         2-k, & {\rm if}\ S\cong N_k.
        \end{cases}
\]

Any graph $G$ can be embedded on some surface $S$, that is, it can be drawn on $S$ with no crossing edges; in addition, one can take the surface $S$ to be either orientable or non-orientable. Denote by $\chi(G)$ the largest integer $\chi$ for which $G$ admits an
embedding on a surface $S$ with $\chi(S)=\chi$. For example, $\chi(G)=2$ means $G$
is planar, while $\chi(G)=1$ means $G$ is non-planar but can be embedded on the projective plane.

Suppose that $G$ is a connected graph which admits an embedding on a surface $S$ whose Euler characteristic $\chi$ is as large as possible, i.e. $\chi(S)=\chi(G)$. By Mohar and Thomassen~\cite[\S3.4]{MoharThomassen}, this embedding of $G$ on $S$ can be taken to be a \emph{2-cell embedding}, meaning that all faces are homeomorphic to an open disk. In this case one has the following result.

\begin{Euler}(c.f. \cite{MoharThomassen})
Suppose that a graph $G$ with vertex set $V(G)$ and edge set $E(G)$
admits a $2$-cell embedding on a surface $S$, and let $F(G)$ be the set
of faces in this embedding. Then
\[
|V(G)|-|E(G)|+|F(G)|=\chi(S).
\]
\end{Euler}

One needs to rewrite Euler's formula in a different form in order to apply it to obtain upper bounds for the bondage number. Every edge $uv$ in the 2-cell embedding of $G$ on $S$ appears on the boundary of either two distinct faces $F\ne F'$ or a unique
face $F=F'$; in the former case $uv$ occurs exactly once on the
boundary of each of the two faces $F$ and $F'$, while in the latter
case $uv$ occurs exactly twice on the boundary of the face $F=F'$.
Let $f(uv)$ and $f'(uv)$ be the number of edges on the boundary of
$F$ and $F'$, whether or not $F$ and $F'$ are distinct.
For instance, a path $P_n$ with $n$ vertices is embedded on a
sphere with only one face, and for any edge in $P_n$ we have
$f(uv)=f'(uv)=2(n-1)$. We may assume that $f(uv)$ and $f'(uv)$ are at least $3$---otherwise one has $G=P_2$ which is a trivial case. There is a weight associated to the edge $uv$, that is,
\[
w(uv)=\frac{1}{d(u)}+\frac{1}{d(v)}-1+\frac{1}{f(uv)}
+\frac{1}{f'(uv)}-\frac{\chi}{|E(G)|}.
\]
This weight is called the \emph{curvature} of $uv$ by Gagarin and Zverovich~\cite{GZ1}. One can rewrite Euler's formula as
\begin{equation}\label{eq:weight}
\sum_{uv\in E(G)} w(uv)=|V(G)|-|E(G)|+|F(G)|-\chi=0.
\end{equation}

The above equation was used by Carlson and Develin~\cite{CarlsonDevelin} in two special cases $\chi=2$ and $\chi=0$, and later used by Gargarin and Zverovich in full generality to establish the following result.

\begin{theorem}[Gagarin and Zverovich~\cite{GZ1}]
\label{thm:GZ}
Let $G$ be a graph embeddable on an orientable surface of
genus $h$ and a non-orientable surface of genus $k$. Then
$b(G)\leq\min\{\Delta(G)+h+2,\Delta(G)+k+1\}$.
\end{theorem}

According to Theorem~\ref{thm:GZ}, if $G$ is planar ($h=0$, $\chi=2$) or
can be embedded on the real projective plane ($k=1$, $\chi=1$), then $b(G)\leq\Delta(G)+2$. For larger values of $h$ and $k$, it was mentioned in \cite{GZ1} that improvements of Theorem~\ref{thm:GZ} can be achieved by adjusting its proof, and an explicit improvement of Theorem~\ref{thm:GZ} was obtained by the first author of this paper.

\begin{theorem}[Huang~\cite{Huang}]\label{thm:H}
Let $G$ be a graph embedded on a surface whose Euler characteristic $\chi$ is as large as possible. If $\chi\leq 0$ then $b(G)\leq\Delta(G)+\lfloor r\rfloor$,
where $r$ is the largest real root of the following cubic equation in $z$:
\[
z^3+2z^2+(6\chi-7)z+18\chi-24 = 0.
\]
A weaker but asymptotically equivalent upper bound is $b(G)\leq\Delta(G)+\lceil\sqrt{12-6\chi}-1/2\rceil$.
\end{theorem}

In Section~\ref{sec:arbitrary} we will prove an upper bound which is not only stronger than the above result but also asymptotically better as $\chi\to -\infty$. We will also establish further upper bounds in Section~\ref{sec:girth}, Section~\ref{sec:order}, and Section~\ref{sec:size} for the bondage number $b(G)$ when the girth, order, or size of the graph $G$ is large, which improve another result of the first auther~\cite{Huang} and the following result of Gagarin and Zverovich~\cite[Corollary~17, Corollary~19]{GZ2}.

\begin{theorem}[Gagarin and Zverovich~\cite{GZ2}]\label{thm:GZ2}
Let G be a connected graph $2$-cell embeddable on an orientable surface of genus $h\geq1$ and a non-orientable surface of genus $k\geq1$. Then

\begin{itemize}
\item
$b(G) \leq \Delta(G) + \lceil \ln^2 h \rceil + 3$ if $n \geq h$,
\item
 $b(G) \leq \Delta(G) + \lceil \ln h \rceil + 3$ if $n \geq h^{1.9}$,
\item
  $b(G) \leq \Delta(G) + 4$ if $n \geq h^{2.5}$,
\item
  $b(G) \leq \Delta(G) + \lceil \ln^2 k \rceil + 2$ if $n \geq k/6$,
\item
$b(G) \leq \Delta(G) + \lceil \ln k \rceil + 3$ if $n \geq k^{1.6}$,
\item
$b(G) \leq \Delta(G) + 3$ if $n \geq k^2$.
\end{itemize}
\end{theorem}

Gagarin and Zverovich~\cite{GZ2} also obtained some constant upper bounds for the bondage number of graphs embedded on surfaces. It seems interesting to look for further improvements of these constant bounds in the future.

\section{The general case}\label{sec:arbitrary}

In this section we establish an upper bound for the bondage number $b(G)$ of a graph $G$ with $\chi(G)\leq 0$, which improves the previously known upper bounds for $b(G)$ in such cases.

\begin{lemma}\label{lem:main}
Let $G$ be a connected graph with $b'(G)\geq\Delta(G)+z$ for some integer $z\geq0$. Let $uv$ be an arbitrary edge of $G$ and write $c(u,v)=|N(u)\cap N(v)|$. Then $\min\{d(u),d(v)\}\geq z+1+c(u,v)$ and
\begin{equation}\label{eq:edge}
|E(G)|\geq |V(G)|(2z+2+c(u,v))/4 \geq (2z+2+c(u,v))^2/4.
\end{equation}
\end{lemma}

\begin{proof}
Assume $d(u)\leq d(v)$, without loss of generality. By (\ref{eq:b'}), one has
\[
\Delta(G)+z\leq b'(G)\leq d(u)+d(v)-1-c(u,v).
\]
Thus
\[
d(u)\geq \Delta(G)-d(v)+z+1+c(u,v)\geq z+1+c(u,v).
\]
It follows that $d(v)\geq d(u)\geq z+1+c(u,v)$, which implies that $v$ has at least $z$ neighbors that are not contained in $\{u\}\cup N(u)$. Hence
\[
|V(G)|\geq 1+(z+1+c(u,v))+z=2z+2+c(u,v).
\]
By (\ref{eq:b'}), one also has
\[
4|E(G)|/|V(G)|-1 \geq b'(G)\geq \Delta(G)+z\geq d(u)+z\geq 2z+1+c(u,v).
\]
Thus (\ref{eq:edge}) holds.
\end{proof}

\begin{lemma}\label{lem:roots}
Let $z\geq0$ and $\chi\leq0$. Then the following inequalities
\begin{equation}\label{ineq1}
A(z)=z^2-2z+2\chi-3>0,
\end{equation}
\begin{equation}\label{ineq2}
B(z)=20z^3+4z^2+3(16\chi-41)z+96\chi-126>0,
\end{equation}
\begin{equation}\label{ineq3}
C(z)=z^3 + z^2 + (3\chi - 8)z  + 9\chi - 12>0
\end{equation}
hold if and only if $z$ is larger than the largest real root $t=t(\chi)$ of $C(z)$. In addition, $t\geq3$ is a decreasing function of $\chi\leq0$.
\end{lemma}

\begin{proof}
We first show that the largest real root $t=t(\chi)$ of $C(z)$ is larger
than or equal to all the real roots of $A(z)$ and $B(z)$,
by using the intermediate value theorem and the limits
\[
\lim_{z\to\infty}A(z)=
\lim_{z\to\infty}B(z)=
\lim_{z\to\infty}C(z)=\infty.
\]

The polynomial $A(z)$ has two roots
\[
z_1=1+\sqrt{4-2\chi}>0 \quad \text{and} \quad
z_2=1-\sqrt{4-2\chi}<0.
\]
Substituting $z_1$ in $C(z)$ gives
\[
C(z_1)= (\chi + 1)\sqrt{4-2\chi} + 4\chi - 2
\]
which is negative if $\chi\leq -1$ and is $0$ if $\chi=0$. By the intermediate value theorem,
$C(z)$ has a real root larger than or equal to $z_1$. Thus $t\geq z_1>z_2$.

Next consider $B(z)$. If $\chi=0$ then $B(z)=(5z-14)(2z+3)^2$, $C(z)=(z-3)(z+2)^2$, and thus $t=3$ is larger than the two roots $14/5$ and $-3/2$ of $B(z)$. Assume $\chi\leq -1$ below. Then $B(3)=240\chi+81<0$. Applying the intermediate value theorem to $B(z)$ gives the existence of real root(s) of $B(z)$ in $(3,\infty)$; let $z_3$ be the largest one. Then
\begin{eqnarray*}
B(z_3)-16C(z_3) &=& 4z_3^3-12z_3^2+5z_3+66-48\chi \\
&=& z_3(2z_3-1)(2z_3-5) + 66-48\chi >0
\end{eqnarray*}
which implies $C(z_3)<0$. Again by the intermediate value
theorem, $C(z)$ has a root larger than $z_3$, and thus
$t>z_3>3$.

Therefore $t=t(\chi)\geq3$ is larger than or equal to all the real roots of
$A(z)$ and $B(z)$ for all $\chi\leq 0$. It follows that $A(z)$, $B(z)$,
and $C(z)$ are all positive whenever $z>t$; otherwise the
intermediate value theorem would imply that $A(z)$, $B(z)$,
or $C(z)$ has a root larger than $t$, a contradiction.

Conversely, assume that $A(z)$, $B(z)$, and $C(z)$ are all positive and we need to show that $z>t$. Suppose to the contrary that $0\leq z\leq t$. It is clear that $z\ne t$ since $C(z)>0=C(t)$. Thus $z<t$ and there exists a point $s$ in $(z,t)$ such that
\[
C'(s)=3s^2 + 2s + 3\chi - 8 <0
\]
by the mean value theorem. Then
\[
C'(s)-3A(s) = 8s-3\chi+1>0
\]
implies $A(s)<0$. We have seen that the upward parabola $A(z)$ has two roots
$z_1>0$ and $z_2<0$. Since $A(z)>0$ and $z\geq0$, one has $z>z_1$. Then $s>z$ implies $A(s)>0$, a contradiction. Hence $z>t$.

Finally, we show that $t(\chi)\geq 3$ is a decreasing function of $\chi\leq 0$. For any $\epsilon>0$, one has
\[
C(z;\chi)-C(z;\chi-\epsilon) = (3z+9)\epsilon.
\]
This implies
\[
C(t(\chi);\chi-\epsilon)=-(3t(\chi)+9)\epsilon<0.
\]
By the
intermediate value theorem, $C(z;\chi-\epsilon)$ has a real root
larger than $t(\chi)$, and thus its largest real root
$t(\chi-\epsilon)$ is also larger than $t(\chi)$.
\end{proof}

\begin{lemma}\label{lem1}
Let $G$ be a connected graph embedded on a surface whose Euler characteristic $\chi$ is as large as possible. Assume $\chi\leq 0$ and let the largest real root of $z^3 + z^2 + (3\chi - 8)z  + 9\chi - 12$ be $z=t=t(\chi)$. Then $b'(G)\leq\Delta(G)+\lfloor t\rfloor$.
\end{lemma}

\begin{proof}
Since $b'(G)$ is an integer, it suffices to show that $b'(G)<\Delta(G)+z$ for any integer $z>t$. Suppose to the contrary that $b'(G)\geq \Delta(G)+z$ for some integer $z>t$. By Lemma~\ref{lem:roots}, the inequalities (\ref{ineq1}), (\ref{ineq2}), and (\ref{ineq3}) all hold since $z>t$.  Let $uv$ be an arbitrary edge in $G$. Assume $d(u)\leq d(v)$ and $f(uv)\leq f'(uv)$, without loss of generality. Let $c(u,v)=|N(u)\cap N(v)|$. By Lemma~\ref{lem:main}, one has
\[
d(v)\geq d(u) \ge z+1+c(u,v),
\]
\[
|E(G)|\geq (2z+2+c(u,v))^2/4.
\]

If $c(u,v)=0$ then $d(v)\geq d(u)\geq z+1$, $|E(G)|\geq (z+1)^2$, and $f'(uv)\geq f(uv)\geq4$. Thus
\begin{eqnarray*}
w(uv) &\leq & \frac 2{z+1}+\frac14+\frac14 -1
-\frac{\chi}{(z+1)^2} \\
&=& -\frac{z^2-2z+2\chi-3}{2(z+1)^2} <0
\end{eqnarray*}
where the last inequality follows from (\ref{ineq1}).

If $c(u,v)=1$ then $d(v)\geq d(u)\geq z+2$, $|E(G)|\geq (2z+3)^2/4$, $f'(uv)\geq 4$, and $f(uv)\geq 3$. Thus
\begin{eqnarray*}
w(uv) & \leq & \frac2{z+2}+\frac14+\frac13-1-
\frac{4\chi}{(2z+3)^2} \\
&=& -\frac{20z^3+4z^2+3(16\chi-41)z+96\chi-126}
{12(z+2)(2z+3)^2}<0
\end{eqnarray*}
where the last inequality follows from (\ref{ineq2}).

If $c(u,v)\geq 2$ then $d(v)\geq d(u)\geq z+3$, $|E(G)|\geq (z+2)^2$, and $f'(uv)\geq f(uv)\geq 3$. Thus
\begin{eqnarray*}
w(uv)&\leq& \frac2{z+3}+\frac13+\frac13-1
-\frac{\chi}{(z+2)^2}\\
&=&-\frac{z^3 + z^2 + (3\chi - 8)z  + 9\chi - 12}{3(z+3)(z+2)^2}<0
\end{eqnarray*}
where the last inequality follows from (\ref{ineq3}).

Therefore $w(uv)<0$ for all edges $uv$ in $G$. This contradicts Equation (\ref{eq:weight}). It follows that $b'(G)<\Delta(G)+z$ for any integer $z>t$, which implies $b'(G)\leq\Delta(G)+\lfloor t\rfloor$.
\end{proof}

\begin{theorem}\label{thm1}
Let $G$ be a graph embedded on a surface whose Euler characteristic $\chi$ is as large as possible. Assume $\chi\leq 0$ and let the largest real root of $z^3 + z^2 + (3\chi - 8)z  + 9\chi - 12$ be $z=t=t(\chi)$. Then $b(G)\leq\Delta(G)+\lfloor t\rfloor$.
\end{theorem}

\begin{proof}
If $G$ is connected then Lemma~\ref{lem:HR}, Theorem~\ref{thm:ad}, and Lemma~\ref{lem1} imply $b(G)\leq b'(G)\leq \Delta(G)+\lfloor t\rfloor$. If $G$ has multiple components $G_1,\ldots,G_\ell$, then $\chi\leq\chi_i=\chi(G_i)$ for all $i$, since an embedding of $G$ on a surface $S$ automatically includes an embedding of $G_i$ on $S$. By definition, $b(G)=\min\{b(G_1),\ldots,b(G_\ell)\}$. If $\chi_i>0$ for some $i$ then Theorem~\ref{thm:GZ} implies
\[
b(G)\leq b(G_i)\leq \Delta(G_i)+2 <\Delta(G)+ 3\leq \Delta(G)+\lfloor t \rfloor.
\]
Assume $\chi_i\leq 0$ for all $i=1,\ldots,\ell$. By Lemma~\ref{lem:roots}, $\chi\leq\chi_i$ implies $t(\chi_i)\leq t(\chi)$. Hence
\[
b(G)\leq b(G_i)\leq\Delta(G_i)+\lfloor t(\chi_i)\rfloor
\leq \Delta(G)+\lfloor t(\chi)\rfloor.
\]
This completes the proof.
\end{proof}

\begin{corollary}\label{cor1}
Let $G$ be a graph embedded on a surface whose Euler characteristic $\chi$ is as large as possible. If $\chi\leq 0$ then $b(G)\leq\Delta(G)+1+\lfloor \sqrt{4-3\chi} \rfloor$.
\end{corollary}

\begin{proof}
If $\chi=0$ then $1+\sqrt{4-3\chi}=3=t(\chi=0)$. If $\chi\leq-1$ then $1+\sqrt{4-3\chi}>t(\chi)$ by  Lemma~\ref{lem:roots}  and the following inequialities:
\begin{eqnarray*}
A(1+\sqrt{4-3\chi}) & = &  -\chi >0, \\
B(1+\sqrt{4-3\chi}) &=& (25-12\chi)\sqrt{4-3\chi} +31 - 48\chi > 0,\\
C(1+\sqrt{4-3\chi}) &=& \sqrt{4-3\chi}-2 >0.
\end{eqnarray*}
Hence the result follows immediately from Theorem~\ref{thm1}.
\end{proof}

\begin{remark}

Computations in \textsf{Sage} give the following formula
\[
t(\chi) =  \frac 13\left(D + (25-9\chi)/D -1  \right)
\]
where $D=\left( 9\sqrt{9\chi^3 + 69\chi^2 - 125\chi} - 108\chi + 125\right)^{\frac 13}$. Note that this formula works in $\mathbb C$, the field of complex numbers. Athough Theorem~\ref{thm1} is stronger than Corollary~\ref{cor1}, asymptotically they are equivalent by the following limit
\[
\lim_{\chi\to-\infty} t(\chi)/(1+\sqrt{4-3\chi})=1.
\]
One can check that Theorem~\ref{thm1} and Corollary~\ref{cor1} are both stronger than the previous result of Theorem~\ref{thm:H}. Asymptotically, the improvement is by a factor of $\sqrt 2$, as shown by the limit
\[
\lim_{\chi\to-\infty} \frac{1/2+ \sqrt{12-6\chi}}{1+\sqrt{4-3\chi}} =\sqrt 2.
\]
\end{remark}

\begin{example}
We give a table below to compare Theorem~\ref{thm1}  with Theorem~\ref{thm:H} for $-21\leq \chi\leq0$.
\[
\begin{tabular}{|c|ccccccccccc|}
\hline
$\chi$ & 0 & -1 & -2 & -3 & -4 & -5 & -6 & -7 & -8 & -9 & -10  \\
\hline
$\lfloor r \rfloor$ & 3 & 3 & 4 & 5 & 5 & 6 & 6 & 7 & 7 & 8 & 8\\
\hline
$\lfloor t\rfloor$ & 3 & 3 & 4 & 4 & 4 & 5 & 5 & 5 & 6 & 6 & 6\\
\hline\hline
$\chi$ &-11&-12&-13& -14 & -15 & -16 & -17 & -18 & -19 & -20 & -21\\
\hline
$\lfloor r\rfloor$ & 8 & 9 & 9 & 9 & 10 & 10 & 10 & 11 & 11 & 11 &11\\
\hline
$\lfloor t \rfloor$ &7 & 7 & 7 & 7 & 7 & 8 & 8 & 8 & 8 & 8 & 9\\
\hline
\end{tabular}
\]
\end{example}

\begin{remark}
We do not know whether the upper bound for $b(G)$ given in Theorem~\ref{thm1} is sharp for all graphs $G$ with $\chi(G)\leq 0$, or whether the weaker result for $b'(G)$ given by Lemma~\ref{lem1}  is sharp for connected graphs $G$ with $\chi(G)\leq 0$.
\end{remark}

\section{Graphs with large girth}\label{sec:girth}

Our result can be further improved when the graph $G$ has
large \emph{girth} $g(G)$, defined as the length of the
shortest cycle in $G$. If $G$ has no cycle then we set $g(G)=\infty$ by convention, and in such case one has $b(G)\leq 2$ by \cite{BHNS}.

\begin{theorem}\label{thm:girth}
Let $G$ be a graph embedded on a surface whose Euler characteristic
$\chi$ is as large as possible. If $\chi\leq0$ and $g=g(G)<\infty$,
then $b(G)\leq \Delta(G)+\lfloor s \rfloor$
where $s$ is the larger root of the quadratic polynomial
$A(z)=(g-2)z^2-4z+\chi g-g-2$, i.e.
\[
s=\frac{2+\sqrt{g^2-g(g-2)\chi}}{(g-2)}.
\]
\end{theorem}

\begin{proof}
Assume $G$ is connected for the same reason as argued in the proof of
Theorem~\ref{thm1}. It suffices to show that $b(G)<\Delta(G)+z$ for
any positive integer $z$ satisfying $A(z)>0$.

Suppose to the contrary that $b(G)\geq \Delta(G)+z$ for some positive integer $z$ satisfying $A(z)>0$. Let $uv$ be an arbitrary edge of $G$. By Lemma~\ref{lem:main}, $\min\{d(u),d(v)\}\geq z+1$ and $|E(G)|\geq (z+1)^2$. One also has $f(uv)\geq g$ and $f'(uv)\geq g$ by definition. Hence
\begin{eqnarray*}
w(uv) &\leq & \frac 2{z+1}+\frac2g-1-\frac{\chi}{(z+1)^2} \\
&=& -\frac{(g-2)z^2-4z+\chi g-g-2}{g(z+1)^2}<0
\end{eqnarray*}
where the last inequality follows from $A(z)>0$.
This contradicts Equation (\ref{eq:weight}).
\end{proof}

The first author~\cite[Proposition~10]{Huang} showed that, with the same conditions as Theorem~\ref{thm:girth},
\[
b(G)\leq \Delta(G)+ \left\lfloor \frac{\sqrt{8g(2-g)\chi+(3g-2)^2}-(g-6)}{2(g-2)} \right\rfloor.
\]
It is not hard to check that Theorem~\ref{thm:girth} improves this result.

\begin{example}
One has
\[
b(G)\leq \Delta(G)+
\begin{cases}
2, & {\rm if}\ \chi=0,\ g\geq 5,\\
1, & {\rm if}\ \chi=0,\ g\geq 7,\\
2, & {\rm if}\ \chi=-1,\ g\geq 5,\\
3, & {\rm if}\ \chi=-2,\ g\geq 4, \\
2, & {\rm if}\ \chi=-2,\ g\geq 6.
\end{cases}
\]
\end{example}

\begin{corollary}\label{cor:girth}
Let $G$ be a triangle-free graph embedded on a surface whose Euler characteristic $\chi$ is as large as possible. If $\chi\leq0$ then $b(G)\leq \Delta(G)+1+ \lfloor \sqrt{4-2\chi} \rfloor$.
\end{corollary}

\begin{proof}
One can check that the upper bound for $b(G)$ provided by the previous proposition is a decreasing function of $g\geq3$. Hence taking $g=4$ gives the desired result.
\end{proof}

\begin{remark}
We do not know whether the upper bounds for $b(G)$ given by Theorem~\ref{thm:girth} and Corollary~\ref{cor:girth} are sharp, but one can check that they are actually upper bounds for $b'(G)$ as long as $G$ is connected. When is $G$ connected and triangle-free, Corollary~\ref{cor:girth} implies
\begin{equation}\label{eq:b'g4}
b'(G)\leq \Delta(G)+1+\sqrt{4-2\chi}
\end{equation}
This bound is indeed sharp. For example, let $G$ be the complete bipartite graph $K_{n,n}$, which is triangle-free. One sees that $b'(G)=2n-1$ and $\Delta(G)=n$. One also has $\chi(G)={(4n-n^2)}/2$ by Ringel~\cite{Ringel1,Ringel2}. Hence the quality in (\ref{eq:b'g4}) holds for $G=K_{n,n}$. On the other hand, for $G=K_{n,n}$ with $n\geq 2$ one can check that $\gamma(G)=2$ and $b(G)=n<2n-1$. So it is not clear whether the upper bound $b(G)\leq \Delta(G)+1+\sqrt{4-2\chi}$ is sharp for triangle-free graphs.
\end{remark}

\section{Connected graphs with large order}\label{sec:order}

The order $|V(G)|$ of a connected nontrivial graph $G$ has the following lower bound in terms of its Euler characteristic $\chi(G)$.

\begin{proposition}[\cite{GZ2}]\label{prop:order}
Let $G$ be a connected graph with $n=|V(G)|\geq 2$ embedded on a surface whose Euler characteristic $\chi$ is as large as possible. Then $n\geq(3+\sqrt{17-8\chi})/2$.
\end{proposition}

In this section we assume $|V(G)|\geq -\chi$ and obtain asymptotically better upper bounds for the bondage number $b(G)$.

\begin{lemma}
Let $\chi\leq0$ and $n\geq1$. Then the following inequalities
\begin{equation}\label{ineq4}
A(z)=nz-3n+4\chi>0,
\end{equation}
\begin{equation}\label{ineq5}
B(z)=10nz^2 - (13n - 48\chi)z - 42n + 96\chi>0,
\end{equation}\begin{equation}\label{ineq6}
C(z)=nz^2 - (n - 6\chi)z - 6n + 18\chi>0
\end{equation}
are all valid if and only if
$
z>  1/2 - {3\chi}/{n} + \sqrt{25/4 - 21\chi/n + 9\chi^2/n^2}.
$
\end{lemma}

\begin{proof}
The graph of $A(z)$ is an upward straight line with a unique root $a = 3 - 4\chi/n\geq 3$. The graph of $B(z)$ is an upward parabola with two roots
\[
b = \frac{1}{20n}(13n - 48\chi + \sqrt{\beta}),\quad b' = \frac{1}{20n}(13n - 48\chi - \sqrt{\beta})
\]
where $\beta = (43n)^2 - 5088n\chi + (48\chi)^2$. One sees that
\[
 43n-48\chi \leq \sqrt\beta \leq 43n-2544\chi/43.
\]
Hence $b'<0<b\leq 2.8-5.4\chi/n$. The graph of $C(z)$ is an upward parabola with two roots
\[
c=\frac{1}{2n}(n - 6\chi + \sqrt{\gamma}),\quad c'=\frac{1}{2n}(n - 6\chi - \sqrt{\gamma})
\]
where $\gamma=25n^2 - 84n\chi + 36\chi^2$. One sees that $\gamma\geq 5n-6\chi$. Hence $c'<0< c$ and
\[
c\geq 3-6\chi/n \geq \max\{a,b\}.
\]
Hence $A(z)$, $B(z)$, and $C(z)$ are all positive if and only if $z>c$.
\end{proof}

\begin{theorem}\label{thm:order}
Let $G$ be a connected graph embedded on a surface whose Euler characteristic $\chi$ is as large as possible. Let $n=|V(G)|$ and assume $\chi\leq0$. Then $b(G) \leq b'(G)\leq \Delta(G) + \lfloor c \rfloor$ where
\[
c = 1/2 - {3\chi}/{n} + \sqrt{25/4 - 21\chi/n + 9\chi^2/n^2}.
\]
\end{theorem}

\begin{proof}
It suffices to show that $b'(G)< \Delta(G)+z$ for any integer $z>c$. Suppose to the contrary that $b'(G)\geq \Delta(G)+z$ for some integer $z>c$. Let $uv$ be an arbitrary edge in $G$. Assume $d(u)\leq d(v)$ and $f(uv)\leq f'(uv)$, without loss of generality. Let $|N(u)\cap N(v)|=c(u,v)$. By Lemma~\ref{lem:main}, $d(u)\geq z+1+c(u,v)$ and $|E(G)|\geq n(2z+2+c(u,v))/4$.

If $c(u,v)=0$ then $d(v)\geq d(u)\geq z+1$, $f'(uv)\geq f(uv)\geq4$, and $|E(G)|\geq n(z+1)/2$. Thus
\begin{eqnarray*}
w(uv) &\leq & \frac 2{z+1}+\frac14+\frac14 -1
-\frac{2\chi}{n(z+1)} \\
&=& -\frac{nz-3n+4\chi}{2n(z+1)} <0
\end{eqnarray*}
where the last inequality follows from (\ref{ineq4}).

If $c(u,v)=1$ then $d(v)\geq d(u)\geq z+2$, $f'(uv)\geq 4$, $f(uv)\geq 3$, and $|E(G)|\geq n(2z+3)/4$. Thus
\begin{eqnarray*}
w(uv) & \leq & \frac2{z+2}+\frac14+\frac13-1-
\frac{4\chi}{n(2z+3)} \\
&=& -\frac{10nz^2 - (13n - 48\chi)z - 42n + 96\chi}
{12n(z+2)(2z+3)}<0
\end{eqnarray*}
where the last inequality follows from (\ref{ineq5}).

If $c(u,v)\geq 2$ then $d(v)\geq d(u)\geq z+3$, $f'(uv)\geq f(uv)\geq 3$, and $|E(G)|\geq n(z+2)/2$. Thus
\begin{eqnarray*}
w(uv)&\leq& \frac2{z+3}+\frac13+\frac13-1
-\frac{2\chi}{n(z+2)}\\
&=&-\frac{nz^2 - (n - 6\chi)z - 6n + 18\chi}{3n(z+2)(z+3)}<0
\end{eqnarray*}
where the last inequality follows from (\ref{ineq6}).

Therefore $w(uv)<0$ for all edges $uv$ in $G$.
This contradicts Equation (\ref{eq:weight}).
\end{proof}

\begin{corollary}\label{cor:order}
Let $G$ be a connected graph embedded on a surface whose Euler characteristic $\chi$ is as large as possible. Suppose that $\chi\leq0$ and $n=|V(G)|$. Then
\begin{itemize}
\item
$b(G)\leq b'(G)\leq \Delta(G)+9$ if $n\geq-\chi$,
\item
$b(G)\leq b'(G)\leq \Delta(G)+6$ if $n\geq -2\chi$,
\item
$b(G)\leq b'(G)\leq \Delta(G)+5$ if $n\geq -3\chi$,
\item
$b(G)\leq b'(G)\leq \Delta(G)+4$ if $n\geq -4\chi$,
\item
$b(G)\leq b'(G)\leq \Delta(G)+3$ if $n\geq -8\chi$.
\end{itemize}
\end{corollary}

\begin{proof}
Suppose that $n\geq -d \chi$ for some $d>0$. Then
\[
c = \frac 12 - \frac{3\chi}{n} + \sqrt{ \frac{25}{4} - \frac{21\chi}{n} + \frac{9\chi^2}{n^2} } \leq \frac 12+\frac3d + \sqrt{\frac{25}4+\frac{21}{d}+\frac{9}{d^2}}
\]
and Theorem~\ref{thm:order} implies that
\[
b(G)\leq \Delta(G)+\lfloor c \rfloor \leq \Delta(G) + \left\lfloor \frac 12+\frac3d + \sqrt{\frac{25}4+\frac{21}{d}+\frac{9}{d^2}} \right\rfloor.
\]
Taking $d=1,2,3,4,8$ gives the desired upper bounds.
\end{proof}

Theorem~\ref{thm:order} and Corollary~\ref{cor:order} asymptotically improve a result of Gagarin and Zverovich~\cite[Corollary~17, Corollary~19]{GZ2} (see Theorem~\ref{thm:GZ2}).

\section{Connected graphs with large size}\label{sec:size}

Using Euler's formula, Proposition~\ref{prop:order}, and the fact that $|F(G)|\geq1$, one obtains a lower bound
\[
|E(G)| \geq \frac52-\chi+\frac12\sqrt{17-8\chi(G)}
\]
for the size of a connected nontrivial graph $G$ in terms of its Euler characteristic $\chi(G)$. In this section we assume $|E(G)|>-3\chi(G)$ and obtain better upper bounds for the bondage number $b(G)$.

\begin{lemma}
Let $\chi\leq0$ and $m>-3\chi$. Then the following inequalities
\begin{equation}\label{ineq7}
A(z)=(m +2\chi)z - 3m +2\chi >0,
\end{equation}
\begin{equation}\label{ineq8}
B(z)=(5m + 12\chi)z - 14m + 24\chi >0,
\end{equation}\begin{equation}\label{ineq9}
C(z)= (m + 3\chi)z - 3m + 9\chi >0
\end{equation}
are all valid if and only if $z>3 - 18\chi/(m+3\chi)$.
\end{lemma}

\begin{proof}
Since $m>-3\chi$, one sees that $A(z)$, $B(z)$, and $C(z)$ are all upward straight lines whose roots are
\[
a=3 - \frac{8\chi}{m + 2\chi},\quad b = \frac{14}5 - \frac{288\chi}{5(5m+12\chi)},\quad c = 3 - \frac{18\chi}{m+3\chi}.
\]
One can check that
\[
c-a = \frac{-2\chi(5m+6\chi)}{(m+2\chi)(m+3\chi)}>0,
\]
\[
c-b = \frac{(m+3\chi)(m-12\chi)-18m\chi}{(5m+12\chi)(m+3\chi)}>0.
\]
Hence $A(z)$, $B(z)$, and $C(z)$ are all positive if and only if $z>c$.
\end{proof}

\begin{theorem}\label{thm4}
Let $G$ be a connected graph embedded on a surface whose Euler characteristic $\chi$ is as large as possible. Suppose that $m=|E(G)|>-3\chi\geq0$. Then $b(G) \leq b'(G) \leq \Delta(G) + \lfloor c \rfloor$ where  $c = 3 - 18\chi/(m+3\chi)$.
\end{theorem}

\begin{proof}
It suffices to show that $b'(G)< \Delta(G)+z$ for any integer $z>c$. Suppose to the contrary that $b'(G)\geq \Delta(G)+z$ for some integer $z>c$. Let $uv$ be an arbitrary edge in $G$. Assume $d(u)\leq d(v)$ and $f(uv)\leq f'(uv)$, without loss of generality. Let $|N(u)\cap N(v)|=c(u,v)$. By Lemma~\ref{lem:main}, one has $d(u)\geq z+1+c(u,v)$.

If $c(u,v)=0$ then $d(v)\geq d(u)\geq z+1$ and $f'(uv)\geq f(uv)\geq4$. Thus
\begin{eqnarray*}
w(uv) &\leq & \frac 2{z+1}+\frac14+\frac14 -1 -\frac\chi m \\
&=& -\frac{(m +2\chi)z - 3m +2\chi}{2m(z + 1)} <0
\end{eqnarray*}
where the last inequality follows from (\ref{ineq7}).

If $c(u,v)=1$ then $d(v)\geq d(u)\geq z+2$, $f'(uv)\geq 4$, $f(uv)\geq 3$. Thus
\begin{eqnarray*}
w(uv) & \leq & \frac2{z+2}+\frac14+\frac13-1 -\frac\chi m\\
&=& -\frac{(5m + 12\chi)z - 14m + 24\chi}{12m(z + 2)} <0
\end{eqnarray*}
where the last inequality follows from (\ref{ineq8}).

If $c(u,v)\geq 2$ then $d(v)\geq d(u)\geq z+3$, and $f'(uv)\geq f(uv)\geq 3$. Thus
\begin{eqnarray*}
w(uv)&\leq& \frac2{z+3}+\frac13+\frac13-1 - \frac\chi m\\
&=& -\frac{(m + 3\chi)z - 3m + 9\chi}{3m(z + 3)} <0
\end{eqnarray*}
where the last inequality follows from (\ref{ineq9}).

Therefore $w(uv)<0$ for any edge $uv$ in $G$.
This contradicts Equation (\ref{eq:weight}).
\end{proof}

\begin{corollary}
Let $G$ be a connected graph embedded on a surface whose Euler characteristic $\chi$ is as large as possible. Suppose that $\chi\leq 0$ and $m=|E(G)|$.  Then

\begin{itemize}
\item
$b(G) \leq b'(G) \leq \Delta(G)+8$ if $m> -6\chi$,
\item
$b(G)\leq  b'(G) \leq \Delta(G)+7$ if $m> -6.6\chi$,
\item
$b(G)\leq b'(G) \leq \Delta(G)+6$ if $m> -7.5\chi$,
\item
$b(G)\leq b'(G) \leq \Delta(G)+5$ if $m> -9\chi$,
\item
$b(G)\leq b'(G) \leq \Delta(G)+4$ if $m> -12\chi$,
\item
$b(G)\leq b'(G) \leq \Delta(G)+3$ if $m> -21\chi$.
\end{itemize}
\end{corollary}

\begin{proof}
This follows immediately from the above theorem.
\end{proof}


\end{document}